\documentclass[12pt]{article}%
\usepackage{amsfonts}
\usepackage{amsmath}
\usepackage{amssymb}
\usepackage{graphicx}%
\providecommand{\U}[1]{\protect\rule{.1in}{.1in}}
\newtheorem{theorem}{Theorem}

\newtheorem{definition}[theorem]{Definition}

\newtheorem{proposition}[theorem]{Proposition}

\newenvironment{proof}[1][Proof]{\noindent\textbf{#1.} }{\ \rule{0.5em}{0.5em}}
\begin{document}

\title{\textbf{Average operators on rectangular Herz spaces\thanks{To appear in Tatra
Mountains Mathematical Publications}}}
\author{Carolina Espinoza-Villalva
\and Martha Guzm\'{a}n-Partida}
\date{}
\maketitle

\begin{abstract}
We introduce a family of Herz type spaces considering rectangles instead of
balls and we study continuity properties of some average operators acting on them.

\bigskip



\end{abstract}

\section{Introduction}

Herz spaces have been studied for many years. The roots of this subject lie on
the pioneering work of N. Wiener \cite{wiener1}, A. Beurling \cite{beurling}
and C. Herz \cite{herz}. Later, these spaces were generalized by other
mathematicians in order to study continuity properties of classical operators
in harmonic analysis, as well as to develop local versions of Hardy spaces and
bounded mean oscillation spaces.

There are several definitions of Herz space. The following is classical and
corresponds to the inhomogeneous setting: a measurable function $f$ belongs to
the Herz space $K_{p,q}^{\alpha}\left(  \mathbb{R}^{n}\right)  $, $1\leq
p,q<\infty$, $\alpha\in\mathbb{R}$ if%
\begin{equation}
\left\Vert f\right\Vert _{K_{p,q}^{\alpha}}:=\left(
{\displaystyle\sum\limits_{k=0}^{\infty}}
2^{nk\alpha q}\left\Vert f\chi_{C_{k}}\right\Vert _{p}^{q}\right)
^{1/q}<\infty\text{,} \label{classic1}%
\end{equation}
and for $q=\infty$%
\begin{equation}
\left\Vert f\right\Vert _{K_{p,\infty}^{\alpha}}:=\sup_{k\geq0}\left(
2^{nk\alpha}\left\Vert f\chi_{C_{k}}\right\Vert _{p}\right)  <\infty\text{.}
\label{classic2}%
\end{equation}
Here $C_{0}$ is the open unit ball $B_{1}\left(  0\right)  $ and
$C_{k}=B_{2^{k}}\left(  0\right)  \setminus B_{2^{k-1}}\left(  0\right)  $,
$k\in\mathbb{N}$.

Setting $\alpha=-1/p$ in (\ref{classic2}) we obtain the space $B^{p}\left(
\mathbb{R}^{n}\right)  $ that also can be characterized by mean of the
condition (\cite{Feichtinger}, \cite{GCuerva})%
\begin{equation}
\sup_{R\geq1}\left(  \frac{1}{\left\vert B_{R}\left(  0\right)  \right\vert }%
{\displaystyle\int\nolimits_{B_{R}\left(  0\right)  }}
\left\vert f\left(  x\right)  \right\vert ^{p}dx\right)  ^{1/p}<\infty
\label{classic3}%
\end{equation}
and the quantity on the left hand side of (\ref{classic3}) defines an
equivalent norm to $\left\Vert f\right\Vert _{K_{p,\infty}^{-1/p}}$ that is
usually denoted by $\left\Vert f\right\Vert _{B^{p}}$. With any of these norms
$B^{p}\left(  \mathbb{R}^{n}\right)  $ turns out to be a Banach space.
Moreover, for $1\leq p_{1}<p_{2}<\infty$ we have the inclusions $B^{p_{2}%
}\left(  \mathbb{R}^{n}\right)  \subset B^{p_{1}}\left(  \mathbb{R}%
^{n}\right)  $ and $L^{\infty}\left(  \mathbb{R}^{n}\right)  \subset
B^{p}\left(  \mathbb{R}^{n}\right)  $ for every $p$.

In this work we will restrict to the context of the space $B^{p}\left(
\mathbb{R}^{n}\right)  $ for $1\leq p<\infty$. Our aim is to explore what
happens when we consider rectangles\ with sides parallel to the coordinate
axes instead of balls in (\ref{classic3}). As we will see below, although we
obtain a smaller space than $B^{p}\left(  \mathbb{R}^{n}\right)  $, it is
still appropriate to study continuity properties of some classical operators.
In the context of the present paper, we study continuity properties of some
discrete and continuous versions of the classical Hardy average operator. This
operator has been extensively studied for many authors on different function
spaces. We restrict ourself to consider the most simple versions of this
operator in order to make easy the reading of the present paper.

The manuscript is organized as follows: the second section is devoted to
introduce the rectangular Herz spaces and to give some examples. In the third
section we introduce the average operators to be considered and we prove the
continuity of these averages on our spaces.

We will employ standard notation along this work and we will also adopt the
convention to denote by $C$ a constant that could be changing line by line.

\section{Rectangular Herz spaces}

For $1\leq p<\infty$, we define the following space%
\[
\mathcal{B}^{p}\left(  \mathbb{R}^{n}\right)  =\left\{  f\in L_{loc}%
^{p}\left(  \mathbb{R}^{n}\right)  :\left\Vert f\right\Vert _{\mathcal{B}^{p}%
}<\infty\right\}  \text{,}%
\]
where
\begin{equation}
\left\Vert f\right\Vert _{\mathcal{B}^{p}}:=\sup_{\substack{R_{j}%
\geq1\\j=1,...,n}}\left[  \frac{1}{R_{1}...R_{n}}%
{\displaystyle\int\nolimits_{\left[  -R_{1},R_{1}\right]  \times
...\times\left[  -R_{n},R_{n}\right]  }}
\left\vert f\left(  x\right)  \right\vert ^{p}dx\right]  ^{1/p}\text{.}
\label{norma1}%
\end{equation}
If the context does not cause confusion, we will simply write $\mathcal{B}%
^{p}$. Notice that for $n=1$, the spaces $\mathcal{B}^{p}\left(
\mathbb{R}\right)  $ and $B^{p}\left(  \mathbb{R}\right)  $ coincide.

Standard arguments (see \cite{A-GP-L}, for example) allow us to see that
$\left(  \mathcal{B}^{p},\left\Vert \cdot\right\Vert _{\mathcal{B}^{p}%
}\right)  $ is a Banach space. Moreover, it is clear that $\mathcal{B}%
^{p}\subset B^{p}$ and $\left\Vert \cdot\right\Vert _{B^{p}}\leq\left\Vert
\cdot\right\Vert _{\mathcal{B}^{p}}$ since Lebesgue measure of balls and cubes
are comparable.

\begin{proposition}
\label{inclusion}The space $\mathcal{B}^{p}\left(  \mathbb{R}^{n}\right)  $ is
properly contained in $B^{p}\left(  \mathbb{R}^{n}\right)  $ when $n\geq2$.
\end{proposition}

\begin{proof}
For the sake of clarity, we will consider the case $n=2$.

Let $f:\mathbb{R}^{2}\rightarrow\mathbb{R}$ be the function defined as
follows:
\[
f\left(  x\right)  =\left\{
\begin{array}
[c]{ccl}%
0 & \text{if} & x\notin\left(  \left[  -1,1\right]  \times\mathbb{R}\right)
\cup\left(  \mathbb{R}\times\left[  -1,1\right]  \right)  ,\\
1 & \text{if} & x\in\left[  -1,1\right]  \times\left[  -1,1\right]  ,\\
2^{1/p} & \text{if} & x\in\left(  \left[  -1,1\right]  \times\left(
1,2\right]  \right) \\
& \quad & \cup\left(  \left[  -1,1\right]  \times\left[  -2,-1\right)  \right)
\\
& \quad & \cup\left(  \left(  1,2\right]  \times\left[  -1,1\right]  \right)
\\
&  & \cup\left(  \left[  -2,-1\right)  \times\left[  -1,1\right]  \right)  ,\\
& \cdot & \\
& \cdot & \\
& \cdot & \\
n^{1/p} & \text{if} & x\in\left(  \left[  -1,1\right]  \times\left(
n-1,n\right]  \right) \\
& \quad & \cup\left(  \left[  -1,1\right]  \times\left[  -n,-n+1\right)
\right) \\
& \quad & \cup\left(  \left(  n-1,n\right]  \times\left[  -1,1\right]  \right)
\\
& \quad & \cup\left(  \left[  -n,-n+1\right)  \times\left[  -1,1\right]
\right)  ,\text{ }n\geq2\text{.}%
\end{array}
\right.
\]

Take $R\geq1$. We can find $k\in\mathbb{N}$ such that $k\leq R<k+1$ and thus%
\begin{align*}
\frac{1}{\left\vert \left[  -R,R\right]  ^{2}\right\vert }%
{\displaystyle\int\nolimits_{\left[  -R,R\right]  ^{2}}}
\left\vert f\left(  x\right)  \right\vert ^{p}dx  &  \leq\frac{1}{4k^{2}}%
{\displaystyle\int\nolimits_{\left[  -\left(  k+1\right)  ,k+1\right]  ^{2}}}
\left\vert f\left(  x\right)  \right\vert ^{p}dx\\
&  =\frac{1}{4k^{2}}\left[  1\,.\,2^{2}+2\,.\,2^{3}+3\,.\,2^{3}+...+\left(
k+1\right)  .\,2^{3}\right] \\
&  \leq\frac{2}{k^{2}}\left[  1+2+...+\left(  k+1\right)  \right] \\
&  =\frac{\left(  k+1\right)  \left(  k+2\right)  }{k^{2}}\leq6
\end{align*}
which shows that $f\in B^{p}\left(  \mathbb{R}^{2}\right)  $. However, if now
we consider rectangles of the form $\left[  -1,1\right]  \times\left[
-L,L\right]  $ for $L\geq2$, we can pick $m\in\mathbb{N}$ such that $m\leq
L<m+1$ and therefore%
\begin{align*}
\frac{1}{\left\vert \left[  -1,1\right]  \times\left[  -L,L\right]
\right\vert }%
{\displaystyle\int\nolimits_{\left[  -1,1\right]  \times\left[  -L,L\right]
}}
\left\vert f\left(  x\right)  \right\vert ^{p}dx  &  =\frac{1}{4L}%
{\displaystyle\int\nolimits_{\left[  -1,1\right]  \times\left[  -L,L\right]
}}
\left\vert f\left(  x\right)  \right\vert ^{p}dx\\
&  \geq\frac{1}{4\left(  m+1\right)  }%
{\displaystyle\int\nolimits_{\left[  -1,1\right]  \times\left[  -m,m\right]
}}
\left\vert f\left(  x\right)  \right\vert ^{p}dx\\
&  =\frac{1}{4\left(  m+1\right)  }\left[  1\,.\,2^{2}+2\,.\,2^{2}%
+...+m\,.\,2^{2}\right] \\
&  =m/2\rightarrow\infty\text{ if }m\rightarrow\infty\text{,}%
\end{align*}
that is, $f\notin\mathcal{B}^{p}\left(  \mathbb{R}^{2}\right)  $.
\end{proof}

Using the idea of the previous example we can get a characterization of the
space $\mathcal{B}^{p}\left(  \mathbb{R}^{n}\right)  $. To this end, consider
the following subsets of $\mathbb{R}^{n}$:%
\[
C_{j_{1},j_{2},...,j_{n}}=C_{j_{1}}\times C_{j_{2}}\times...\times C_{j_{n}}%
\]
where%
\[
C_{0}=\left[  -1,1\right]  \text{ and }C_{j}=\left\{  x\in\mathbb{R}%
:2^{j-1}<\left\vert x\right\vert \leq2^{j}\right\}
\]
for $j\in\mathbb{N}$.

For $1\leq p<\infty$ and $f\in L_{loc}^{p}\left(  \mathbb{R}^{n}\right)  $
define%
\[
\left\Vert f\right\Vert _{\mathcal{B}^{p}}^{\ast}:=\sup_{\substack{j_{i}%
\geq0\\i=1,2,...,n}}2^{-\frac{\left(  j_{1}+j_{2}+...+j_{n}\right)  }{p}%
}\left\Vert f\chi_{C_{j_{1},j_{2},...,j_{n}}}\right\Vert _{p}\text{.}%
\]
Now, we can state the following characterization.

\begin{proposition}
\label{characterization}$f\in\mathcal{B}^{p}\left(  \mathbb{R}^{n}\right)  $
if and only if $\left\Vert f\right\Vert _{\mathcal{B}^{p}}^{\ast}<\infty$.
Moreover, $\left\Vert f\right\Vert _{\mathcal{B}^{p}}$ and $\left\Vert
f\right\Vert _{\mathcal{B}^{p}}^{\ast}$ are equivalent norms.
\end{proposition}

\begin{proof}
Assume that $\left\Vert f\right\Vert _{\mathcal{B}^{p}}^{\ast}<\infty$. For
$i=1,...,n$ let $R_{i}>1$ and choose $j_{i}\in\mathbb{N}$ such that
\[
2^{j_{i}-1}<R_{i}\leq2^{j_{i}}\text{.}%
\]
We have that
\begin{align*}%
{\displaystyle\int\nolimits_{{\prod\nolimits_{i=1}^{n}}\left[  -R_{i}%
,R_{i}\right]  }}
\left\vert f\left(  x\right)  \right\vert ^{p}dx  &  \leq%
{\displaystyle\sum\limits_{k_{1}=0}^{j_{1}}}
\,%
{\displaystyle\sum\limits_{k_{2}=0}^{j_{2}}}
...%
{\displaystyle\sum\limits_{k_{n}=0}^{j_{n}}}
\int\nolimits_{C_{k_{1},k_{2},...,k_{n}}}\left\vert f\left(  x\right)
\right\vert ^{p}dx\\
&  \leq%
{\displaystyle\sum\limits_{k_{1}=0}^{j_{1}}}
\,%
{\displaystyle\sum\limits_{k_{2}=0}^{j_{2}}}
...%
{\displaystyle\sum\limits_{k_{n}=0}^{j_{n}}}
2^{k_{1}+k_{2}+...+k_{n}}\left(  \left\Vert f\right\Vert _{\mathcal{B}^{p}%
}^{\ast}\right)  ^{p}\\
&  \leq C2^{j_{1}+j_{2}+...+j_{n}}\left(  \left\Vert f\right\Vert
_{\mathcal{B}^{p}}^{\ast}\right)  ^{p}\\
&  \leq CR_{1}R_{2}...R_{n}\left(  \left\Vert f\right\Vert _{\mathcal{B}^{p}%
}^{\ast}\right)  ^{p}\text{.}%
\end{align*}
Hence $f\in\mathcal{B}^{p}\left(  \mathbb{R}^{n}\right)  $ and $\left\Vert
f\right\Vert _{\mathcal{B}^{p}}\leq C\left\Vert f\right\Vert _{\mathcal{B}%
^{p}}^{\ast}$.

Conversely, if $f\in\mathcal{B}^{p}\left(  \mathbb{R}^{n}\right)  $,
$i=1,...,n$ and $j_{i}\geq0$%
\begin{align*}
\left\Vert f\chi_{C_{j_{1},j_{2},...,j_{n}}}\right\Vert _{p}^{p}  &  =%
{\displaystyle\int\nolimits_{{\prod\nolimits_{i=1}^{n}}\left[  -2^{j_{i}%
},2^{j_{i}}\right]  }}
\left\vert f\left(  x\right)  \right\vert ^{p}dx\\
&  \leq C\left\Vert f\right\Vert _{\mathcal{B}^{p}}^{p}2^{j_{1}+j_{2}%
+...+j_{n}}%
\end{align*}
which implies that
\[
\left\Vert f\right\Vert _{\mathcal{B}^{p}}^{\ast}=\sup_{\substack{j_{i}%
\geq0\\i=1,2,...,n}}2^{-\frac{\left(  j_{1}+j_{2}+...+j_{n}\right)  }{p}%
}\left\Vert f\chi_{C_{j_{1},j_{2},...,j_{n}}}\right\Vert _{p}\leq C\left\Vert
f\right\Vert _{\mathcal{B}^{p}}\text{.}%
\]
This concludes the proof.
\end{proof}

\section{Continuity of average operators}

Average integral operators were considered by Hardy, Littlewood and P\'{o}lya
in \cite{Hardy}. They proved the following classical inequality:%
\[%
{\displaystyle\int\nolimits_{0}^{\infty}}
\left(  \frac{F\left(  x\right)  }{x}\right)  ^{p}dx\leq\left(  \frac{p}%
{p-1}\right)  ^{p}%
{\displaystyle\int\nolimits_{0}^{1}}
f^{p}\left(  x\right)  dx,
\]
where $1<p<\infty$, $F\left(  x\right)  =\int_{0}^{x}f\left(  t\right)  dt,$
$f\geq0$ and the constant $\left(  \frac{p}{p-1}\right)  ^{p}$ is the best possible.

Closely related to this operator is the operator $H_{\varphi}$ introduced by
Carton-Lebrun and Fosset in \cite{carton} and by Xiao in \cite{Xiao} which is
pointwisely defined as follows:
\begin{equation}
H_{\varphi}f\left(  x\right)  :=%
{\displaystyle\int\nolimits_{0}^{1}}
f\left(  tx\right)  \varphi\left(  t\right)  dt. \label{hardy average}%
\end{equation}

Xiao in \cite{Xiao} proved continuity of $H_{\varphi}$ under appropriate
conditions on $\varphi$ on $L^{p}\left(  \mathbb{R}^{n}\right)  $ and
$BMO\left(  \mathbb{R}^{n}\right)  $ for $1\leq p\leq\infty$. It is our goal
to prove continuity of this and other related operators in our rectangular
Herz spaces.

We will start by considering the following discrete version of
(\ref{hardy average}).

Let $\left\{  r_{k}\right\}  _{k=1}^{\infty}$ be a sequence in $\left(
0,1\right]  $ which is strictly decreasing and $\lim_{k\rightarrow\infty}%
r_{k}=0$. If $f:\mathbb{R}^{n}\longrightarrow\mathbb{R}$ is a Lebesgue
measurable function and $\varphi:\left\{  r_{k}:k\in\mathbb{N}\right\}
\longrightarrow\left(  0,\infty\right)  $ is any function, consider the
operator $H_{\varphi}^{d}$ formally defined as%
\[
H_{\varphi}^{d}f\left(  x\right)  =%
{\displaystyle\sum\limits_{k=1}^{\infty}}
\varphi\left(  r_{k}\right)  f\left(  r_{k}x\right)  \text{.}%
\]
Now, notice that a necessary and sufficient condition for the existence of
$H_{\varphi}^{d}$ as a bounded operator on $L^{p}\left(  \mathbb{R}%
^{n}\right)  $ is that%
\begin{equation}%
{\displaystyle\sum\limits_{k=1}^{\infty}}
r_{k}^{-n/p}\varphi\left(  r_{k}\right)  <\infty. \label{serie}%
\end{equation}

Indeed, assuming the convergence of the series in (\ref{serie}), given $f\in
L^{p}\left(  \mathbb{R}^{n}\right)  $, $1\leq p<\infty$, and using Minkowski
inequality we obtain%
\begin{align*}
\left\Vert H_{\varphi}^{d}f\right\Vert _{p}  &  \leq%
{\displaystyle\sum\limits_{k=1}^{\infty}}
\varphi\left(  r_{k}\right)  \left(  \int\nolimits_{\mathbb{R}^{n}}\left\vert
f\left(  r_{k}x\right)  \right\vert ^{p}dx\right)  ^{1/p}\\
&  =\left\Vert f\right\Vert _{p}%
{\displaystyle\sum\limits_{k=1}^{\infty}}
r_{k}^{-n/p}\varphi\left(  r_{k}\right)  \text{,}%
\end{align*}
which implies that $\left\Vert H_{\varphi}^{d}\right\Vert _{L^{p}\rightarrow
L^{p}}\leq%
{\displaystyle\sum\limits_{k=1}^{\infty}}
r_{k}^{-n/p}\varphi\left(  r_{k}\right)  $.

Conversely, if $H_{\varphi}^{d}$ is bounded on $L^{p}\left(  \mathbb{R}%
^{n}\right)  $, we can consider as Xiao in \cite{Xiao} the function
\[
f_{\varepsilon}\left(  x\right)  =\left\vert x\right\vert ^{-\frac{n}%
{p}-\varepsilon}\chi_{\left\{  \left\vert x\right\vert >1\right\}  }\text{,}%
\]
where $0<\varepsilon<1$. It turns out that $\left\Vert f_{\varepsilon
}\right\Vert _{p}=\frac{C_{n}}{p\varepsilon}$, $C_{n}$ an $n$-dimensional
constant and%
\[
H_{\varphi}^{d}f_{\varepsilon}\left(  x\right)  =\left(
{\displaystyle\sum\limits_{k=1}^{\infty}}
r_{k}^{-\frac{n}{p}-\varepsilon}\varphi\left(  r_{k}\right)  \right)
\left\vert x\right\vert ^{-\frac{n}{p}-\varepsilon}\chi_{\left\{  \left\vert
x\right\vert >1\right\}  }\text{.}%
\]
Thus, same procedure as done in \cite{Xiao} shows that%
\[
\left\Vert H_{\varphi}^{d}\right\Vert _{L^{p}\rightarrow L^{p}}^{p}\left\Vert
f_{\varepsilon}\right\Vert _{p}^{p}\geq\left[  \varepsilon^{\varepsilon}%
{\displaystyle\sum\limits_{k=1}^{\infty}}
r_{k}^{-\frac{n}{p}-\varepsilon}\varphi\left(  r_{k}\right)  \right]
^{p}\left\Vert f_{\varepsilon}\right\Vert _{p}^{p}%
\]
and therefore%
\[
\left\Vert H_{\varphi}^{d}\right\Vert _{L^{p}\rightarrow L^{p}}\geq\left[
\varepsilon^{\varepsilon}%
{\displaystyle\sum\limits_{k=1}^{\infty}}
r_{k}^{-\frac{n}{p}-\varepsilon}\varphi\left(  r_{k}\right)  \right]
\geq\varepsilon^{\varepsilon}%
{\displaystyle\sum\limits_{k=1}^{\infty}}
r_{k}^{-\frac{n}{p}}\varphi\left(  r_{k}\right)
\]
for any $0<\varepsilon<1$. Now, letting $\varepsilon\rightarrow0$ we obtain%
\[
\left\Vert H_{\varphi}^{d}\right\Vert _{L^{p}\rightarrow L^{p}}\geq%
{\displaystyle\sum\limits_{k=1}^{\infty}}
r_{k}^{-\frac{n}{p}}\varphi\left(  r_{k}\right)  \text{.}%
\]

We have proved the following result.

\begin{theorem}
\label{cdiscrete}The operator $H_{\varphi}^{d}$ is a bounded operator on
$L^{p}\left(  \mathbb{R}^{n}\right)  $, $1\leq p<\infty$, if and only if $%
{\displaystyle\sum\limits_{k=1}^{\infty}}
r_{k}^{-\frac{n}{p}}\varphi\left(  r_{k}\right)  <\infty$. In such case
\[
\left\Vert H_{\varphi}^{d}\right\Vert _{L^{p}\rightarrow L^{p}}=%
{\displaystyle\sum\limits_{k=1}^{\infty}}
r_{k}^{-\frac{n}{p}}\varphi\left(  r_{k}\right)  \text{.}%
\]

\end{theorem}

We can also consider the following generalization of the operator $H_{\varphi
}^{d}$.

Let $\Phi:\left\{  r_{k_{1}}^{\left(  1\right)  }:k_{1}\in\mathbb{N}\right\}
\times...\times\left\{  r_{k_{n}}^{\left(  n\right)  }:k_{n}\in\mathbb{N}%
\right\}  \longrightarrow\left(  0,\infty\right)  $ any function, where for
every $j=1,...,n$, the sequence $\left\{  r_{k_{j}}^{\left(  j\right)
}\right\}  _{k_{j}=1}^{\infty}\subset\left(  0,1\right]  $, is strictly
decreasing and $\lim_{k_{j}\rightarrow\infty}r_{k_{j}}^{\left(  j\right)  }%
=0$. For a Lebesgue measurable function $f:\mathbb{R}^{n}\longrightarrow
\mathbb{R}$ define formally%
\begin{equation}
\mathbb{H}_{\Phi}^{d}f\left(  x\right)  =%
{\displaystyle\sum\limits_{k_{1}=1}^{\infty}}
...%
{\displaystyle\sum\limits_{k_{n}=1}^{\infty}}
\Phi\left(  r_{k_{1}}^{\left(  1\right)  },...,r_{k_{n}}^{\left(  n\right)
}\right)  f\left(  r_{k_{1}}^{\left(  1\right)  }x_{1},...,r_{k_{n}}^{\left(
n\right)  }x_{n}\right)  \text{.} \label{disc hardy average}%
\end{equation}
With the same proof as in Theorem \ref{cdiscrete} we can show:

\begin{theorem}
\label{cgdiscrete}The operator $\mathbb{H}_{\Phi}^{d}$ is a bounded operator
on $L^{p}\left(  \mathbb{R}^{n}\right)  $, $1\leq p<\infty$, if and only if
\[%
{\displaystyle\sum\limits_{k_{1}=1}^{\infty}}
...%
{\displaystyle\sum\limits_{k_{n}=1}^{\infty}}
\Phi\left(  r_{k_{1}}^{\left(  1\right)  },...,r_{k_{n}}^{\left(  n\right)
}\right)  \left(  r_{k_{1}}^{\left(  1\right)  }\right)  ^{-1/p}...\left(
r_{k_{n}}^{\left(  n\right)  }\right)  ^{-1/p}<\infty\text{.}%
\]
In such case%
\[
\left\Vert \mathbb{H}_{\Phi}^{d}\right\Vert _{L^{p}\rightarrow L^{p}}=%
{\displaystyle\sum\limits_{k_{1}=1}^{\infty}}
...%
{\displaystyle\sum\limits_{k_{n}=1}^{\infty}}
\Phi\left(  r_{k_{1}}^{\left(  1\right)  },...,r_{k_{n}}^{\left(  n\right)
}\right)  \left(  r_{k_{1}}^{\left(  1\right)  }\right)  ^{-1/p}...\left(
r_{k_{n}}^{\left(  n\right)  }\right)  ^{-1/p}\text{.}%
\]

\end{theorem}

Now we will study the action of the operator $\mathbb{H}_{\Phi}^{d}$ on our
rectangular Herz spaces defined in the previous section.

For these spaces is even easier the proof of the continuity of the operator
$\mathbb{H}_{\Phi}^{d}$. We provide it for the sake of completeness.

\begin{theorem}
\label{cgdiscreteBp}The operator $\mathbb{H}_{\Phi}^{d}$ is a bounded operator
on $\mathcal{B}^{p}\left(  \mathbb{R}^{n}\right)  $, $1\leq p<\infty$, if and
only if
\begin{equation}%
{\displaystyle\sum\limits_{k_{1}=1}^{\infty}}
...%
{\displaystyle\sum\limits_{k_{n}=1}^{\infty}}
\Phi\left(  r_{k_{1}}^{\left(  1\right)  },...,r_{k_{n}}^{\left(  n\right)
}\right)  <\infty\text{.} \label{serie2}%
\end{equation}
In such case%
\[
\left\Vert \mathbb{H}_{\Phi}^{d}\right\Vert _{\mathcal{B}^{p}\rightarrow
\mathcal{B}^{p}}=%
{\displaystyle\sum\limits_{k_{1}=1}^{\infty}}
...%
{\displaystyle\sum\limits_{k_{n}=1}^{\infty}}
\Phi\left(  r_{k_{1}}^{\left(  1\right)  },...,r_{k_{n}}^{\left(  n\right)
}\right)  \text{.}%
\]

\end{theorem}

\begin{proof}
Assuming condition (\ref{serie2}), taking $R_{j}>1$, $j=1,...,n$, and using
Minkowski inequality we can see that%
\begin{align*}
&  \left[  \frac{1}{R_{1}...R_{n}}%
{\displaystyle\int\nolimits_{\left[  -R_{1},R_{1}\right]  \times
...\times\left[  -R_{n},R_{n}\right]  }}
\left\vert \mathbb{H}_{\Phi}^{d}f\left(  x\right)  \right\vert ^{p}dx\right]
^{1/p}\\
&  \leq%
{\displaystyle\sum\limits_{k_{1}=1}^{\infty}}
...%
{\displaystyle\sum\limits_{k_{n}=1}^{\infty}}
\Phi\left(  r_{k_{1}}^{\left(  1\right)  },...,r_{k_{n}}^{\left(  n\right)
}\right)  \left[  \frac{1}{R_{1}...R_{n}}%
{\displaystyle\int\nolimits_{\left[  -R_{1},R_{1}\right]  \times
...\times\left[  -R_{n},R_{n}\right]  }}
\left\vert f\left(  r_{k_{1}}^{\left(  1\right)  }x_{1},...,r_{k_{n}}^{\left(
n\right)  }x_{n}\right)  \right\vert ^{p}dx\right]  ^{1/p}\\
&  \leq%
{\displaystyle\sum\limits_{k_{1}=1}^{\infty}}
...%
{\displaystyle\sum\limits_{k_{n}=1}^{\infty}}
\Phi\left(  r_{k_{1}}^{\left(  1\right)  },...,r_{k_{n}}^{\left(  n\right)
}\right)  \left\Vert f\right\Vert _{\mathcal{B}^{p}}\text{,}%
\end{align*}
and hence $\left\Vert \mathbb{H}_{\Phi}^{d}\right\Vert _{\mathcal{B}^{p}}\leq%
{\displaystyle\sum\limits_{k_{1}=1}^{\infty}}
...%
{\displaystyle\sum\limits_{k_{n}=1}^{\infty}}
\Phi\left(  r_{k_{1}}^{\left(  1\right)  },...,r_{k_{n}}^{\left(  n\right)
}\right)  .$

Now, if the operator $\mathbb{H}_{\Phi}^{d}$ is bounded on $\mathcal{B}%
^{p}\left(  \mathbb{R}^{n}\right)  $, it is enough to consider the function
$f_{0}\equiv1$ because in such case we easily obtain the required reverse inequality.
\end{proof}

Our next goal is to generalize the operator given by (\ref{disc hardy average}%
). Before to do this, we will define another class of rectangular spaces
closely related to $\mathcal{B}^{p}$.

\begin{definition}
\label{central-rectangular}For $1\leq p<\infty$ we define%
\[
\mathcal{CMO}^{p}\left(  \mathbb{R}^{n}\right)  =\left\{  f\in L_{loc}%
^{p}\left(  \mathbb{R}^{n}\right)  :\left\Vert f\right\Vert _{\mathcal{CMO}%
^{p}}<\infty\right\}  \text{,}%
\]
where%
\begin{equation}
\left\Vert f\right\Vert _{\mathcal{CMO}^{p}}:=\sup_{\substack{R_{j}%
\geq1\\j=1,...,n}}\left[  \frac{1}{R_{1}...R_{n}}%
{\displaystyle\int\nolimits_{\left[  -R_{1},R_{1}\right]  \times
...\times\left[  -R_{n},R_{n}\right]  }}
\left\vert f\left(  x\right)  -f_{R_{1}...R_{n}}\right\vert ^{p}dx\right]
^{1/p}\text{,} \label{cmo}%
\end{equation}
and $f_{R_{1}...R_{n}}$ is the average of $f$ on $\left[  -R_{1},R_{1}\right]
\times...\times\left[  -R_{n},R_{n}\right]  $.
\end{definition}

It is not difficult to show that $\left(  \mathcal{CMO}^{p},\left\Vert
\cdot\right\Vert _{\mathcal{CMO}^{p}}\right)  $ is a Banach space if we
identify functions that differ by a constant almost everywhere on
$\mathbb{R}^{n}$. Also, we obtain an equivalent norm to $\left\Vert
\cdot\right\Vert _{\mathcal{CMO}^{p}}$ if we consider the quantity%
\[
\left\Vert f\right\Vert _{\mathcal{CMO}^{p}}^{\ast}:=\sup_{\substack{R_{j}%
\geq1\\j=1,...,n}}\inf_{a\in\mathbb{R}}\left[  \frac{1}{R_{1}...R_{n}}%
{\displaystyle\int\nolimits_{\left[  -R_{1},R_{1}\right]  \times
...\times\left[  -R_{n},R_{n}\right]  }}
\left\vert f\left(  x\right)  -a\right\vert ^{p}dx\right]  ^{1/p}\text{.}%
\]
This space is the rectangular version of the space $CMO^{p}$ (\cite{Chen-Lau}%
,\cite{GCuerva}) whose elements satisfy the condition%
\[
\sup_{R\geq1}\left[  \frac{1}{\left\vert Q\left(  0,R\right)  \right\vert }%
{\displaystyle\int\nolimits_{Q\left(  0,R\right)  }}
\left\vert f\left(  x\right)  -f_{Q\left(  0,R\right)  }\right\vert
^{p}dx\right]  ^{1/p}<\infty\text{.}%
\]
Here, $Q\left(  0,R\right)  $ denotes the cube centered at $0$ and side length
equal to $R$. Clearly, $\mathcal{B}^{p}\subset\mathcal{CMO}^{p}\subset
CMO^{p}$.

Now, we consider the following operator:

For Lebesgue measurable functions $f:\mathbb{R}^{n}\longrightarrow\mathbb{R}$,
and $\phi:\left[  0,1\right]  ^{n}\longrightarrow\left(  0,\infty\right)  $,
we define%
\begin{equation}
\mathbb{H}_{\phi}f\left(  x\right)  :=%
{\displaystyle\int\nolimits_{\left[  0,1\right]  ^{n}}}
f\left(  t_{1}x_{1},...,t_{n}x_{n}\right)  \phi\left(  t_{1},...,t_{n}\right)
dt_{1}...dt_{n}\text{.} \label{Hardycontinuo1}%
\end{equation}

Observe that same proof as given by Xiao in \cite{Xiao}, shows that
$\mathbb{H}_{\phi}$ is a bounded operator on $L^{p}\left(  \mathbb{R}%
^{n}\right)  $, $1\leq p<\infty$, if and only if
\[%
{\displaystyle\int\nolimits_{\left[  0,1\right]  ^{n}}}
t_{1}^{-1/p}...t_{n}^{-1/p}\phi\left(  t_{1},...,t_{n}\right)  dt_{1}%
...dt_{n}<\infty\text{.}%
\]
We will give equivalent conditions for the boundedness of the operator
$\mathbb{H}_{\phi}$ on the spaces $\mathcal{B}^{p}$ and $\mathcal{CMO}^{p}$.

\begin{theorem}
\label{continuity1}The operator $\mathbb{H}_{\phi}$ is a bounded operator on
$\mathcal{B}^{p}\left(  \mathbb{R}^{n}\right)  $ and $\mathcal{CMO}^{p}\left(
\mathbb{R}^{n}\right)  $, $1\leq p<\infty$, if and only if
\[%
{\displaystyle\int\nolimits_{\left[  0,1\right]  ^{n}}}
\phi\left(  t_{1},...,t_{n}\right)  dt_{1}...dt_{n}<\infty\text{.}%
\]
Moreover%
\begin{equation}
\left\Vert \mathbb{H}_{\phi}\right\Vert _{\mathcal{B}^{p}\rightarrow
\mathcal{B}^{p}}=\left\Vert \mathbb{H}_{\phi}\right\Vert _{\mathcal{CMO}%
^{p}\rightarrow\mathcal{CMO}^{p}}=%
{\displaystyle\int\nolimits_{\left[  0,1\right]  ^{n}}}
\phi\left(  t_{1},...,t_{n}\right)  dt_{1}...dt_{n}\text{.}
\label{normasfinales}%
\end{equation}

\end{theorem}

\begin{proof}
Just for illustration we prove the equivalence for the space $\mathcal{CMO}%
^{p}\left(  \mathbb{R}^{n}\right)  $.

Suppose that the integral in (\ref{normasfinales}) is finite. Then, for
$R_{j}>1$, $j=1,...,n$ and $f\in\mathcal{CMO}^{p}\left(  \mathbb{R}%
^{n}\right)  $ we can easily see that%
\[
\left(  \mathbb{H}_{\phi}f\right)  _{R_{1}...R_{n}}=%
{\displaystyle\int\nolimits_{\left[  0,1\right]  ^{n}}}
f_{t_{1}R_{1}...t_{n}R_{n}}\phi\left(  t_{1},...,t_{n}\right)  dt_{1}%
...dt_{n}\text{.}%
\]
Now, by Minkowski inequality and an appropriate change of variable we have
that%
\begin{align*}
&  \left[  \frac{1}{R_{1}...R_{n}}%
{\displaystyle\int\nolimits_{\left[  -R_{1},R_{1}\right]  \times
...\times\left[  -R_{n},R_{n}\right]  }}
\left\vert \mathbb{H}_{\phi}f\left(  x\right)  -\left(  \mathbb{H}_{\phi
}f\right)  _{R_{1}...R_{n}}\right\vert ^{p}dx\right]  ^{1/p}\\
&  \leq%
{\displaystyle\int\nolimits_{\left[  0,1\right]  ^{n}}}
\left(  \frac{1}{R_{1}...R_{n}}%
{\displaystyle\int\nolimits_{\left[  -R_{1},R_{1}\right]  \times
...\times\left[  -R_{n},R_{n}\right]  }}
\left\vert f\left(  t_{1}x_{1},...,t_{n}x_{n}\right)  -f_{t_{1}R_{1}%
...t_{n}R_{n}}\right\vert ^{p}dx\right)  ^{1/p}\\
&  \times\phi\left(  t_{1},...,t_{n}\right)  dt_{1}...dt_{n}\\
&  \leq\left\Vert f\right\Vert _{\mathcal{CMO}^{p}}%
{\displaystyle\int\nolimits_{\left[  0,1\right]  ^{n}}}
\phi\left(  t_{1},...,t_{n}\right)  dt_{1}...dt_{n}\text{,}%
\end{align*}
which implies that%
\[
\left\Vert \mathbb{H}_{\phi}\right\Vert _{\mathcal{CMO}^{p}\rightarrow
\mathcal{CMO}^{p}}\leq%
{\displaystyle\int\nolimits_{\left[  0,1\right]  ^{n}}}
\phi\left(  t_{1},...,t_{n}\right)  dt_{1}...dt_{n}\text{.}%
\]
For the converse, it suffices to consider the function $f_{0}\left(  x\right)
\equiv1$.
\end{proof}

Finally, it should be remarked that Theorems \ref{cgdiscreteBp} and
\ref{continuity1} remain true if we consider homogeneous versions of the
spaces $\mathcal{B}^{p}$ and $\mathcal{CMO}^{p}$, that is, those defined by
taking $R_{j}>0$ for every $j=1,...,n$ in (\ref{norma1}) and (\ref{cmo}).

\bigskip

Departamento de Matem\'{a}ticas

Universidad de Sonora

Rosales y Luis Encinas

Hermosillo, Sonora, 83000, M\'{e}xico

Email: carolina.espinoza@mat.uson.mx

\qquad\quad martha@mat.uson.mx\footnote{Corresponding author}


\begin{thebibliography}{99}                                                                                               %


\bibitem {A-GP-L}J. Alvarez, M. Guzm\'{a}n-Partida, J. Lakey, Spaces of
bounded $\lambda$-central mean oscillation, Morrey spaces, and $\lambda
$-central Carleson measures, \emph{Collect. Math. }\textbf{51}, 1 (2000), 1-47.

\bibitem {beurling}A. Beurling, Construction and analysis of some convolution
algebras, \emph{Ann. Inst. Fourier (Grenoble)} \textbf{14} (1964), 1-32.

\bibitem {carton}C. Carton-Lebrun, M. Fosset, Moyennes et quotients de Taylor
dans $BMO$, \emph{Bull. Soc. Roy. Sci. Li\`{e}ge }\textbf{53} (2) (1984), 85-87.

\bibitem {Chen-Lau}Y. Chen, K. Lau, Some new classes of Hardy spaces, \emph{J.
Funct. Anal.} \textbf{84} (1989), 255-278.

\bibitem {Feichtinger}H. Feichtinger, An elementary approach to Wiener's third
Tauberian theorem on the Euclidean $n$-space, \emph{Proceedings, Conference at
Cortona} 1984, Symposia Mathematica 29 (New York, Academic Press, 1987), 267-301.

\bibitem {Feichtinger-Weisz}H. Feichtinger, F. Weisz, Herz spaces and
summability of Fourier transforms, \emph{Math. Nachr.} \textbf{281}, 3 (2008), 309-324.

\bibitem {GCuerva}J. Garc\'{\i}a-Cuerva, Hardy spaces and Beurling algebras,
\emph{J. London Math. Soc. }\textbf{39}, 2\textbf{ }(1989), 499-513.

\bibitem {Hardy}G. Hardy, J.E. Littlewood, G. P\'{o}lya, \emph{Inequalities},
Cambridge University Press, 1999.

\bibitem {herz}C. Herz, Lipschitz spaces and Bernstein's theorem on absolutely
convergent Fourier transforms, \emph{J. Appl. Math. Mech.} \textbf{18} (1968), 283-324.

\bibitem {Xiao}J. Xiao, $L^{p}$ and $BMO$ bounds of weighted Hardy-Littlewood
averages, \emph{J. Math. Anal. Appl}. \textbf{262} (2001), 660-666.

\bibitem {wiener1}N. Wiener, Generalized Harmonic Analysis, \emph{Acta Math.}
\textbf{55} (1930), 117-258.
\end{thebibliography}
\end{document}